\documentclass{amsart}

\usepackage{enumitem}
\setenumerate{nolistsep}
\setitemize{nolistsep}
\usepackage[square, numbers, comma, sort&compress]{natbib}
\usepackage[final]{pdfpages}
\usepackage{graphicx}
\usepackage{esint}
\usepackage{latexsym}
\usepackage{hyperref}
\usepackage{cleveref}
\usepackage{autonum}
\usepackage{mathtools}

\usepackage{amsmath}
\usepackage{amssymb}

\usepackage{lipsum}

\newcommand{\C}{{\mathbb C}}       
\newcommand{\R}{{\mathbb R}}       
\newcommand{\N}{{\mathbb N}}       %
\newcommand{\Z}{{\mathbb Z}}       

\usepackage{tikz}
\usetikzlibrary{graphs, angles, quotes,patterns,hobby}
\usepackage{pgfplots}
\pgfplotsset{compat=1.6}

\DeclareMathOperator{\supp}{supp}

\DeclareMathOperator{\dist}{dist}

\theoremstyle{plain}
\newtheorem{teo}{Theorem}

\theoremstyle{plain}
\newtheorem*{theor*}{Theorem}

\theoremstyle{plain}
\newtheorem{prop}{Proposition}

\theoremstyle{remark}

\newtheorem*{rem*}{Remark}

\theoremstyle{definition}

\theoremstyle{plain}

\theoremstyle{remark}

\theoremstyle{plain}
\newtheorem{lemm}{Lemma}

\numberwithin{equation}{section}

\usepackage{hyperref}
\hypersetup{
    colorlinks=true,
    linkcolor=blue,
    filecolor=blue,      
    urlcolor=blue,
    citecolor=red,
}

\begin{document}


\title[Counterexample]{$L^2$-bounded singular integrals on a purely unrectifiable set in $\R^d$}



\author{Joan Mateu}

\address{Joan Mateu
\\
Departament de Matem\`atiques
\\
Universitat Aut\`onoma de Barcelona
\\
08193 Bellaterra (Barcelona), Catalonia.
}
\email{mateu@mat.uab.cat}

\author{Laura Prat}

\address{Laura Prat
\\
Departament de Matem\`atiques
\\
Universitat Aut\`onoma de Barcelona
\\
08193 Bellaterra (Barcelona), Catalonia.
}
\email{laurapb@mat.uab.cat}

\thanks{All the authors were partially supported by 2017-SGR-0395 (Generalitat de Catalunya). J. M. was also partially supported by MTM2016-75390 (MINECO, Spain). 
L. P. was also partially supported by MTM-2016-77635-P (MINECO, Spain). }

\begin{abstract}
We construct an example of a purely unrectifiable measure $\mu$ in $\R^d$ for which the singular integrals associated to the kernels $\displaystyle{K(x)=\frac{P_{2k+1}(x)}{|x|^{2k+d}}}$, with  $k\geq 1$ and $P_{2k+1}$ a homogeneous harmonic polynomial of degree $2k+1$, are bounded in $L^2(\mu)$. This contrasts starkly with the results concerning the Riesz kernel $\displaystyle{\frac{x}{|x|^{d}}}$ in $\R^d$.
\end{abstract}

 \maketitle

\section{Introduction}
The purpose of this note is to give an example of a purely unrectifiable measure $\mu$ in $\R^d$ for which some singular integrals are bounded in $L^2(\mu)$. Let $\mu$ be a finite measure and consider the singular integral operator $T$ associated with the kernel $K:\R^d\setminus\{0\}\rightarrow\R^d$, so that $$T\mu(x)=\int K(x-y)d\mu(y),$$ when $x$ is away from $\mathrm{supp}(\mu)$.   
Given a function $f\in L^1_{loc}(\mu)$,
we set also
\begin{equation}\label{eq:Tfdef}
T_\mu f(x) = T(f\,\mu)(x) = \int K(x-y) f(y)\,d\mu(y),
\end{equation}
and, for $\varepsilon>0$, we consider the $\varepsilon$-truncated version
$$
T_{\varepsilon}\mu(x) = \int_{|x-y|>\varepsilon} K(x-y) \,d\mu(y).
$$
We also write $T_{\mu,\varepsilon} f(x) = T_\varepsilon (f\mu)(x)$. We say that the operator $T_\mu$ is bounded in $L^2(\mu)$ if the operators $T_{\mu,\varepsilon}$ are bounded in $L^2(\mu)$ uniformly on $\varepsilon>0$.

There is a well-known problem in harmonic analysis, the David-Semmes problem, that deals with the connection between singular integral operators and rectifiability.  Recall that a set $E\subset \R^d$ is called $n$-rectifiable if there are Lipschitz maps $f_i:\mathbb{R}^{n}\to\R^d$, $i=1,2,\ldots$, such that
\begin{equation}
\mathcal{H}^n \Big( E\setminus\bigcup_i f_i(\mathbb{R}^n)\Big)=0.
\end{equation}A set $F\subset\R^d$ is called purely $n$-unrectifiable if $\mathcal{H}^n(F\cap E)=0$ for every $n$-rectifiable set $E$. As for sets, one can define a notion of rectifiabilty for measures: a measure $\mu$ is said to be $n$-rectifiable if it vanishes outside an $n$-rectifiable set $E\subset\R^d$ and, moreover, it is absolutely continuous with respect to $\mathcal{H}^n|_{E}$ (the $n$-dimensional Hausdorff measure restricted to the set $E$).  A measure $\mu$ is said to have $(d-1)$-growth if there exists $C>0$ with $\mu(B(x,r))\leq Cr^{d-1}$ for all closed balls $B(x,r)\subset\R^d$.

In what follows, $\mu$ will be a $(d-1)$-dimensional measure, that is, a finite measure with $\mathcal{H}^{d-1}(\supp\mu)<\infty$ and $(d-1)$-growth.
This paper is concerned with the David-Semmes problem, more concretely with the problem of {\bf characterizing the kernels $K$ such that the boundedness of $T_\mu$ in $L^2(\mu)$ implies the $(d-1)-$rectifiability of the measure $\mu$}. 

In 1999, David and L\'eger \cite{leger} showed that the Cauchy kernel $K(z)=\frac 1 z$, $z\in\mathbb{C}\setminus\{0\}$,  is one of the kernels for which the property in bold holds (for $d=2$), and it even holds for its coordinate parts $x/|z|^2$ and $y/|z|^2$,  $z=x+iy\notin\C\setminus\{0\}$. This was proven by using the connection between
Menger curvature and a special subtle positivity property of the  Cauchy kernel (see also \cite{mmv}) . In \cite{cmpt} the result is extended to the kernels $x^{2n-1}/|z|^{2n}$, $n\in\N$, so there are other examples of $1$-dimensional homogeneous convolution kernels whose $L^2(\mu)-$boundedeness also implies the rectifiability of $\mu$, again because of the symmetrization method (see also \cite{chousionisprat} for the extension of the result to any dimension $d$). It is also worth mentioning the paper \cite{chunaev} where other kernels with the above-mentioned property are given.

The case in the David-Semmes problem when $K$ is the codimension $1$ Riesz kernel $K(x)=x/|x|^d$, $x\in\R^d\setminus\{0\}$, was solved by Nazarov, Tolsa and Volberg (see  \cite{NTV_acta} and \cite{NTV_publ}) by different methods relying on the harmonicity of the kernel. The analogous result for dimensions $n\in [2,d-2]$ in $\R^d$ remains still  open. 

The David-Semmes problem also makes sense for solutions of elliptic equations and for the elliptic measure. Very recently Prat, Puliatti and Tolsa, \cite{ppt},  extended  the solution of the David-Semmes problem to gradients of single layer potentials, which are the analogues of the Riesz transform in the context of elliptic PDE's.  

On the other hand, in 2001, Huovinen \cite{huo2} gave an example of a purely unrectifiable Ahlfors-David regular set $E$ for which the singular integral associated to the kernel $K(z)=\frac{xy^2}{|z|^2}$, $z=x+iy \in\mathbb{C}\setminus\{0\},$ is $L^2(\mathcal H^1_{|E})-$bounded.  He also proved that the principal values of the associated singular integral operator exist $\mathcal H^1_{|E}-$almost everywhere. In 2013,  Jaye  and Nazarov \cite{jayenazarov} showed that for the kernel $K(z)=\frac{\bar z}{z^2}, z\in\mathbb{C}\setminus\{0\}$, the above-mentioned in bold property does not hold either. Moreover, for the measure $\mu$ constructed in \cite{jayenazarov}, they showed that $T_\mu1$ fails to exist
in the sense of principal value $\mu$-almost everywhere. 

Very recently, in \cite{jayemerchan}, sharp sufficient conditions on a (locally
finite, non-negative Borel) measure $\mu$ are given, that ensure the existence of a Calder\'on-Zygmund operator in the principal value sense, provided that the operator is $L^2(\mu)-$bounded.

To state our main result,  let $P_{2k+1}$ be a homogeneous harmonic polynomial of degree $2k+1\geq 3$ in $\R^d$. In this paper we will consider the family of kernels 
\begin{equation}\label{kerneldef}
K(x)=\frac{P_{2k+1}(x)}{|x|^{d+2k}},\;\;x\in\R^d\setminus\{0\}.
\end{equation}
So, from now on, the kernel $K$ will be fixed as in \eqref{kerneldef} and $T_\mu$ will denote its associated operator. 
Taking into account that the class of kernels $K$ for which the boundedness of the operator $T_\mu$ in $L^2(\mu)$ implies the rectifiability of $\mu$ does not change if one replaces the $L^2(\mu)$-condition by  the finiteness of $\|T_\mu(1)\|_{L^\infty(\R^d\setminus\supp(\mu))}<\infty$ (see \cite{ntv} for example), we can now state the main result of this paper. It reads as follows:
 
 \begin{teo}\label{maind}
There exists a $(d-1)$-dimensional purely unrectifiable probability measure $\mu$ with $\|T_{\mu}(1)\|_{L^{\infty}(\R^d\setminus\supp(\mu))}<\infty$.
\end{teo}

It is worth pointing out that if one takes a homogeneous polynomial $Q_{2n+1}$ of degree $2n+1$, $n\in\N$, and decomposes it into spherical harmonics, then one gets
\begin{equation}\label{decomp}
\frac{Q_{2n+1}(x)}{|x|^{d+2n}}=\frac{P_{2n+1}(x)+P_{2n-1}|x|^2+\cdots+P_1(x)|x|^{2n}}{|x|^{d+2n}}=\sum_{k=0}^n\frac{P_{2k+1}(x)}{|x|^{d+2k}},
\end{equation}
with $P_{2k+1}$ being homogeneous harmonic polynomials of degree $2k+1$. Notice that if the homogeneous polynomial $Q_{2n+1}$ we consider has $P_1\equiv 0$ in the decomposition \eqref{decomp}, then theorem \ref{maind} applies also to the operator associated to the kernel $\displaystyle{\frac{Q_{2n+1}(x)}{|x|^{d+2n}},\;\;x\in\R^d\setminus\{0\}}$.
Observe that, in dimension 2, for example, due to the result in \cite{cmpt}, the $L^2(\mu)-$boundedness of the singular integral operator associated to the kernel $x^3/|z|^4$, $z\in\C\setminus\{0\}$, implies the rectifiability of the measure $\mu$. On the other hand, this kernel $x^3/|z|^4$ is not an admissible kernel for theorem \ref{maind} because 
$$4\frac{x^3}{|z|^4}=\frac{x^3-3xy^2}{|z|^4}+\frac{3x}{|z|^2},$$
which means that in the decomposition \eqref{decomp}  the term $P_1$ is not zero. \newline

The paper is organized as follows, the next section is devoted to the proof of two lemmata that turn to be very useful for the construction. In section 3 we construct the purely unrectifiable Cantor type set and the measure $\mu$. This construction and the structure of the proof follow the paper \cite{jayenazarov}, but we have extended it to our family of $(d-1)-$dimensional  kernels in $\R^d$. In section 4, we finish the proof of theorem \ref{maind}. We include an appendix with a different proof of the reflectionless property (i.e. lemma \ref{reflectionless}) in dimension 2. The proof in the appendix is a computation based on lemma 3.1 in \cite{jayenazarov} but adapted to our family of kernels, while the proof of lemma \ref{reflectionless} in section \ref{section1} uses the Fourier transform and holds for any dimension $d$. 


\section{Two useful lemmata}\label{section1}
We first introduce some notation. Let $m_d$ denote the $d$-dimensional Lebesgue measure, normalized so that $m_d(B(0,1))=1$.
Then, if $\kappa_d=\pi^{d/2}/\Gamma(\frac d 2+1)$, a $d$-dimensional cube $Q\subset\R^d$ of sidelength $\ell(Q)=\sqrt[d]{\kappa_d}$ has $m_d(Q)=1$. 

In this section, we will consider the kernel $\displaystyle{K(x)=\frac{P_{2k+1}(x)}{|x|^{d+2k}}}$ for $x\in\R^d\setminus\{0\}$ and $k\geq 1$, where $P_{2k+1}$ is a homogeneous harmonic polynomial of degree $2k+1$. We will show a reflectionless property of the measure $m_d$ and our kernel $K$. It reads as follows. 

\begin{lemm}\label{reflectionless}
Let $x_0\in\R^d$, $r>0$. For any $x\in B(x_0,r)$, $$\int_{B(x_0,r)}K(x-y)dm_d(y)=0.$$
 \end{lemm}
\begin{proof}
 Let $B=B(x_0,r)$. We will show that for any $x\in B$, $\displaystyle{(K*\chi_{B})(x)=0.}$ 
 For this, we claim that for $x\in B$ one can write
 \begin{equation}\label{kernel}
 (K*\chi_B)(x)=c P_{2k+1}(\partial)\left(A_0+A_1|x|^2+\cdots+A_{k+1}|x|^{2k+2}\right),
 \end{equation}
 for some appropriately chosen constants $A_0,\cdots,A_{k+1}$ .
 To show \eqref{kernel} we will compute the Fourier transform of our kernel. It is well known that  (see \cite[p.73 Theorem 5]{stein}) 
 \begin{equation}\label{previous}\widehat{\left(\frac{P_{2k+1}(x)}{|x|^{d+2k}}\right)}(\xi)=c\frac{P_{2k+1}(\xi)}{|\xi|^{2k+2}}=cP_{2k+1}(\xi)\frac1{|\xi|^2}\overset{k+1)}{\cdots}\frac1{|\xi|^2},\end{equation}for some constant $c$. Let 
 \begin{equation}
 E(x)=\left\{
 \begin{array}{l}
 -\frac1{2\pi}\log|x|\;\;\;\;\mbox{for}\; d=2\\\\
 \frac1{d(d-2)\kappa_d}\frac1{|x|^{d-2}}\;\;\;\;\mbox{for}\; d\geq 3
\end{array}\right.                                                                                                                                                                                                                                            \end{equation}be the fundamental solution of the Laplacean, that is $\Delta E=\delta_0$, $\delta_0$ being the Dirac delta at the origin. The expression on the right hand side of \eqref{previous} coincides with the Fourier transform of 
                                                                                                                                                                                \begin{equation}\label{fourier}
c P_{2k+1}(\partial)\left(E*\overset{k+1)}{\cdots}*E\right).
 \end{equation}
 Therefore,  claim \eqref{kernel} will be proved if we can show that for $x\in B$, there exist constants $\tilde A_0,\tilde A_1,\cdots, \tilde A_{k+1}$ such that
 \begin{equation}\label{harmonic}
 \big(E*\overset{k+1)}{\cdots}*E*\chi_B\big)(x)=\tilde A_0+\tilde A_1|x|^2+\cdots+\tilde A_{k+1}|x|^{2k+2}.
 \end{equation}
 We will show \eqref{harmonic} by induction on $k$. Notice that one can choose $A_1$ in order that $$\Delta\left(\big(E*\chi_B\big)(x)-A_1|x|^2\right)=0,\;\;\;x\in B.$$
 Hence, this difference of functions, is a radial harmonic function and by the maximum principle, it is constant. Choosing $A_0$ appropriately, one proves the first step in the induction process, namely that for $x\in B,$ $$\big(E*\chi_B\big)(x)=A_0+A_1|x|^2.$$
 Assume that there exist constants $A_0,\cdots, A_k$ such that for $x\in B,$
 \begin{equation}\label{induction}
 \big(E*\overset{k)}{\cdots}*E*\chi_B\big)(x)= A_0+A_1|x|^2+\cdots+A_{k}|x|^{2k}.
 \end{equation}
 We want to show \eqref{harmonic}. One can easily choose constants $ \tilde A_1,\cdots, \tilde A_{k+1}$ such that for $x\in B,$
  \begin{equation}
  \begin{split}
  \Delta&\left( \big(E*\overset{k+1)}{\cdots}*E*\chi_B\big)(x)-\tilde A_1|x|^2+\cdots-\tilde A_{k+1}|x|^{2k+2}\right)\\&=
  \big(E*\overset{k)}{\cdots}*E*\chi_B\big)(x)-A_0-A_1|x|^2-\cdots-A_{k}|x|^{2k}=0,
  \end{split}
  \end{equation}
where the last equality is due to the induction hypothesis. Applying again the maximum principle to the radial harmonic function $$\big(E*\overset{k+1)}{\cdots}*E*\chi_B\big)(x)-\tilde A_1|x|^2+\cdots-\tilde A_{k+1}|x|^{2k+2},$$
one can see that this function is constant and hence \eqref{harmonic} holds, by choosing $\tilde A_0$ appropriately. Therefore, claim \eqref{kernel} is shown.

 To complete the proof of the lemma, we are only left to show that 
 \begin{equation}\label{nul}
 P_{2k+1}(\partial)(|x|^{2j})=0,\;\;\;1\leq j\leq k+1.
 \end{equation}
 In fact, \eqref{nul} holds for $1\leq j\leq 2k$. One can find a proof of this fact in \cite[p. 1437]{mov}, but we include it here for completeness.
 Taking the Fourier transform in expression \eqref{nul}, we obtain $$\widehat{P_{2k+1}(\partial)(|x|^{2j})}(\xi)=c_jP_{2k+1}(\xi)\Delta^j\delta_0,$$
 where $c_j$ is a constant depending on $j$. Take a test function $\varphi$. Since $P_{2k+1}$ is harmonic, 
 \begin{equation}
 \begin{split}
  \langle P_{2k+1}(\xi)\Delta^j\delta_0,\varphi(\xi)\rangle&=\langle\Delta^{j-1}\delta_0,\Delta(P_{2k+1}(\xi)\varphi(\xi))\rangle\\&=\langle\Delta^{j-1}\delta_0,2\nabla P_{2k+1}(\xi)\cdot\nabla\varphi(\xi)+P_{2k+1}(\xi)\Delta\varphi(\xi)\rangle\\&=\langle\delta_0,\mathcal D(\xi)\rangle=\mathcal D(0),
  \end{split}
 \end{equation}
where $\mathcal D$ is a linear combination of products of the form $\partial^\alpha\varphi(\xi)\partial^\beta P_{2k+1}(\xi)$, with multi-indices $\beta$ of length $|\beta|\leq j\leq 2k$. So, $\partial^\beta P_{2k+1}$ is a homogeneous polynomial of degree at least $2k+1-j\geq 1$. Hence $\partial^\beta P_{2k+1}(0)=0$, which implies $\mathcal D(0)=0$ and completes the proof of \eqref{nul}.

\end{proof}

The second important lemma for the proof of theorem \ref{maind} is the following one.

\begin{lemm}\label{lemmabo}
 Let $x_0\in\R^d$. Fix $r,R\in(0,1]$ with $r$ much smaller than $R$. Let $Q\subset\R^d$ be a cube centered at $x_0$ with sidelength $\ell(Q)=\sqrt[d]{\kappa_d Rr^{d-1}}$ and let $B=B(x_0,2r)$. Assume that $\nu_1$ and $\nu_2$ are Borel measures with $\supp\nu_1\subset Q$, $\supp\nu_2\subset B$ and $\nu_1(\R^d)=\nu_2(\R^d)$. Then, for any $x\in\R^d$ with $\dist(x,Q)\geq\sqrt[d]{\kappa_dr^{d-1}R}/8$, we have
 \begin{equation}
  \begin{split}
   \left|\int_Q K(x-\xi)d\nu_1(\xi)-\int_B K(x-\xi)d\nu_2(\xi)\right|\leq C\left(\sqrt[d]{Rr^{d-1}}\int_Q\frac{d\nu_ 1(\xi)}{|x-\xi|^d}+r\int_B\frac{d\nu_2(\xi)}{|x-\xi|^d}\right),
  \end{split}
\end{equation}
for some constant $C=C(k,d)$.
\end{lemm}
\begin{proof}
Without loss of generality set $x=0$. We claim that  it is enough to show that for any $\xi\in Q$ and $x$ such that $\dist(x,Q)\geq\sqrt[d]{\kappa_dr^{d-1}R}/8$, we have
\begin{equation}\label{claim2}
|K(x-\xi)-K(x)|\leq C\frac{|\xi|}{|x-\xi|^d},
\end{equation}
for some constant $C$ depending on $k$ and $d$. In what follows, the constants $C$ may depend on the fixed parameters $k$ and $d$, although we don't write this dependence. 
Assume \eqref{claim2} holds. Using that $\nu_1(\R^d)=\nu_2(\R^d)$ and plugging the claimed estimate into the integral we get the desired inequality as follows:
\begin{equation}
  \begin{split}
   \left|\int_Q K(x-\xi)d\nu_1(\xi)-\int_B K(x-\xi)d\nu_2(\xi)\right|&= \left|\int_Q \left(K(x-\xi)-K(x)\right)d(\nu_1-\nu_2)(\xi)\right|\\\leq C&\int_Q\frac{|\xi|}{|x-\xi|^d}d\nu_1(\xi)+
   C\int_B\frac{|\xi|}{|x-\xi|^d}d\nu_2(\xi)\\\leq C&\int_Q\frac{\sqrt[d]{Rr^{d-1}}}{|x-\xi|^d}d\nu_ 1(\xi)+C \int_B\frac{r}{|x-\xi|^d}d\nu_2(\xi).
  \end{split}
\end{equation}

To show claim \eqref{claim2}, write
\begin{equation}
\begin{split}
 |K(x-\xi)&-K(x)|=\frac{\left|P_{2k+1}(x-\xi)|x|^{d+2k}-P_{2k+1}(x)|x-\xi|^{d+2k}\right|}{|x-\xi|^{d+2k}|x|^{d+2k}}\\&
 \leq\frac{\left|P_{2k+1}(x-\xi)-P_{2k+1}(x)\right|}{|x-\xi|^{d+2k}}+\frac{\left|P_{2k+1}(x)\right|\left||x-\xi|^{d+2k}-|x|^{d+2k}\right|}{|x|^{d+2k}|x-\xi|^{d+2k}}\\&\lesssim\frac{|\xi|}{|x-\xi|^d},
\end{split}
\end{equation}
by the mean value theorem and the fact that $|x|\approx |x-\xi|$.
To see this, notice that $|x-\xi|\leq |x|+|\xi|\leq 2|x|$ and
$$|x-\xi|\geq\dist(x,Q)\geq\sqrt[d]{\kappa_dr^{d-1}R}/8$$
Therefore, $$|x|\leq |x-\xi|+|\xi|\leq |x-\xi|+\frac{\sqrt{d}}{2}\sqrt[d]{\kappa_dr^{d-1}R}\leq |x-\xi|\big(1+4\sqrt{d}\big)$$

This shows \eqref{claim2} and the lemma.\end{proof}

\section{Construction of the purely unrectifiable Cantor type set and the measure $\mu$}
We need a lemma telling us how to pack cubes into balls. Let $r,\;R\in(0,\infty)$ such that $\frac{R}{r}\in\N$.
\begin{lemm}\label{dpacking}
One can pack $\left(\frac{R}{r}\right)^{d-1}$ pairwise essentially disjoint cubes of side length $\sqrt[d]{\kappa_dr^{d-1}R}$ into a ball of radius $R\left(1+\sqrt{d}\sqrt[d]{\kappa_d\frac{r^{d-1}}{R^{d-1}}}\right)$. 
\end{lemm}
\begin{proof}
Without loss of generality assume that the ball is centered at the origin. Consider the cubic grid of mesh size $\sqrt[d]{\kappa_dr^{d-1}R}$. Let $Q_1,\cdots,Q_M$ be the cubes intersecting $B(0,R)$. Since the main diagonal of each of these cubes measures $\sqrt{d}\sqrt[d]{\kappa_dr^{d-1}R}$, we have $$R+\sqrt{d}\sqrt[d]{\kappa_dr^{d-1}R}=R\left(1+\sqrt{d}\sqrt[d]{\kappa_d}\sqrt[d]{\frac{r^{d-1}}{R^{d-1}}}\right).$$ Hence, these cubes are clearly contained in the ball $B\left(0,R(1+\sqrt d\sqrt[d]{\kappa_d\frac{r^{d-1}}{R^{d-1}}})\right)$.  Since
$$Mr^{d-1}R=\sum_{j=1}^Mm_d(Q_j)>m_d(B(0,R))=R^d,$$
we have $\displaystyle{M>\displaystyle{\frac{R^{d-1}}{r^{d-1}}}}$.
\end{proof}
The construction of the Cantor set follows the paper \cite{jayenazarov}, where they do the construction on the plane. 
We consider a sequence $\{r_k\}_{k\geq 0}$ that tends to zero quickly and such that:
\begin{enumerate}
 \item $r_0=1$.
 \item $\displaystyle{r_{k+1}<\frac{r_k}{B}}$, for some big constant $B$ (depending on $d$)  to be chosen later.
 \item $\displaystyle{\frac{r_k}{r_{k+1}}\in\N}$ and $\displaystyle{\frac 1{r_k}\in\N}$.
\end{enumerate}
Set $\widetilde B_1^{0}=B(0,1)$. We will construct the set iteratively. Given the $k-$th generation of $1/r_k^{d-1}$ balls $\widetilde B_j^{k}$ of radius $r_k$, we proceed to the $(k+1)$-st generation as follows: for each ball $\widetilde B_j^{k}$ we apply Lemma \ref{dpacking} with $R=r_k$ and $r=r_{k+1}$, so we find 
$(\frac{r_{k}}{r_{k+1}})^{d-1}$ pairwise essentially disjoint cubes $Q_{\ell}^{k+1}$ 
of sidelength $\sqrt[d]{\kappa_d r_{k+1}^{d-1}r_{k}}$ contained in the ball $\big(1+A\sqrt[d]{\frac{r_{k+1}^{d-1}}{r_k^{d-1}}}\big)\widetilde B_j^{k}$, here $A=\sqrt d\sqrt[d]\kappa_d$. Set $\widetilde B_{\ell}^{k+1}=B(z_{\ell}^{(k+1)}, r_{k+1})$, where $z_{\ell}^{(k+1)}$ denotes the center of the cube $Q_{\ell}^{k+1}$. We carry out this process for each ball $\widetilde B_j^{k}$ from the $k$-th generation. In total, we get $1/r_{k+1}^{d-1}$ balls $B_{\ell}^{k+1}$ in the $(k+1)$-st level. We do this for each $k\in\N$.  Set  $\displaystyle{\delta_{k+1}=A\sqrt[d]{\frac{r_{k+1}^{d-1}}{r_k^{d-1}}}}$ and 
$$B_j^{k}=(1+\delta_{k+1})\widetilde B_j^{k}\;\;\mbox{  and } \;\;E^k=\bigcup_{j\geq 1}B_j^k.$$
Notice that: \begin{enumerate}
\item For all $k\geq 0,$ $\displaystyle{\cup_{\ell}Q_{\ell}^{k+1}\subset E^k}$.
\item For each $k\geq 1$, $B_j^k\subset Q_j^k$.
\item For each $k\geq 1$, $\dist(B_j^{k},\partial Q_j^k)\geq\sqrt[d]{\kappa_dr_k^{d-1} r_{k-1}}/4$, choosing $B$ appropriately (depending on $d$).

\begin{equation}
\begin{split}
\dist(B_j^{k},\partial Q_j^k)&=\frac{\sqrt[d]{\kappa_d r_k^{d-1}r_ {k-1}}}2-(1+\delta_{k+1})r_k\\&=\sqrt[d]{r_k^{d-1}r_{k-1}}\left(\frac{\sqrt[d]{\kappa_d}}{2}-\sqrt[d]{\frac{r_k}{r_{k-1}}}-A\sqrt[d]{\frac{r_{k+1}^{d-1}}{r_k^{d-1}}\frac{r_k}{r_{k-1}}}\right)\\&>\sqrt[d]{r_k^{d-1}r_{k-1}}\left(\frac{\sqrt[d]{\kappa_d}}{2}-\sqrt[d]{\frac 1 B}-\frac A B\right)\geq\frac{\sqrt[d]\kappa_d}{4}\sqrt[d]{r_k^{d-1}r_{k-1}},
\end{split}
\end{equation}
\item For each $k\geq 1$ and $i\neq j$, $\dist(B_j^{k},B_i^{k})\geq\sqrt[d]{\kappa_dr_k^{d-1}r_{k-1}}/2$.
\end{enumerate}

Notice that $E^{k+1}\subset E^k$ for each $k\geq 0$. Now, set $$E=\bigcap_{k\geq 0} E^k.$$
\begin{lemm}
If $\sum_k\delta_k^{\frac{1}{d}}<\infty$, the set $E$ defined above is purely $(d-1)$-unrectifiable.
\end{lemm}
\begin{proof}
We will show that $\mathcal H^{d-1}(E\cap F)=0$ for any $(d-1)$-rectifiable set $F$.  Let $F$ be a $(d-1)$-rectifiable set. Then it can be covered by balls $B_j=B(z_j,\frac{1}{8}\sqrt[d]{\kappa_dr_k^{d-1}r_{k-1}})$, $1\leq j\leq N$, satisfying $$\sum_{j=1}^Nr^{d-1}(B_j)=\sum_{j=1}^N\frac{1}{8^{d-1}}(\kappa_dr_k^{d-1}r_{k-1})^{(d-1)/d}\leq \mathcal H^{d-1}(F),$$ i.e. $N\leq 8^{d-1}\mathcal H^{d-1}(F)/(\kappa_dr_k^{d-1}r_{k-1})^{(d-1)/d}$.

Recall that the balls $B^k_j$ have radius $(1+\delta_{k+1})r_k\leq 2 r_k$, if we choose the constants $A$ and $B$ apropriately. Then, since for each $z\in\R^d$ and $k\geq 1$, $B(z_j,\frac{1}{8}\sqrt[d]{\kappa_dr_k^{d-1}r_{k-1}})$ can intersect at most one of the balls $B^k_j$,
\begin{equation}
 \begin{split}
  {\mathcal H}^{d-1}(E\cap F)&\leq\sum_{j=1}^N{\mathcal H}^{d-1}(E\cap B(z_j,\frac{1}{8}\sqrt[d]{\kappa_dr_k^{d-1}r_{k-1}}))\leq 2^{d-1}\sum_{j=1}^Nr_k^{d-1}=2^{d-1}Nr_k^{d-1}\\&\leq 16^{d-1}\sqrt[d]{\frac{r_k^{d-1}}{\kappa_dr_{k-1}^{d-1}}}\mathcal H^{d-1}(F)=\frac{16^{d-1}}{\sqrt[d]{\kappa_d^2}\sqrt d}\mathcal H^{d-1}(F)\delta_k\longrightarrow 0,
\end{split}
\end{equation}
as $k\rightarrow\infty$, since $\sum_k\delta_k^{\frac{1}{d}}<\infty$.
\end{proof}

We define now the measure $\mu$ as follows. Set $$\displaystyle{\mu^k_j=\frac{1}{r_k}\chi_{\tilde B^k_ j}m_d}\;\;\mbox{ and }\;\;\mu^k=\displaystyle{\sum_j\mu_j^k}.$$ Notice that $\supp(\mu^k)\subset E^k$ and $\mu^k(\R^d)=1$ for all $k$. Hence, there exists a subsequence of $\mu_j^k$ that converges weakly to a measure $\mu$ with $\supp(\mu)\subset E$ and $\mu(\R^d)=1$. From the construction, we can deduce that our measures $\mu^k$ satisfy the following properties:

\begin{enumerate}
\item $\supp(\mu^k)\subset\bigcup_jB_j^m$ if $k\geq m$.
\item $\mu^k(B^m_j)=r_m^{d-1}$ for $k\geq m$.
\item The measure $\mu^k$ has $(d-1)$-growth, that is, there exists $C_0>0$ such that $\mu^k(B(z,r))\leq C_0r^{d-1}$ for any $z\in\R^d$, $r>0$ and $k\geq 0$. 
\end{enumerate}

To show the $(d-1)-$growth, notice that for $r\geq1$, the property is clear because $\mu^k$ is a probability measure. If $0<r<1$, then $r\in(r_{m+1},r_m)$ for some $m\in\N$. In case $m\geq k$, we have $r_m\leq r_k$. Hence, the disc $B(z,r)$ intersects at most one $B^k_j$, therefore
$$\mu^k(B(z,r))=\frac 1{r_k}m_d(B(z,r)\cap\tilde B^k_j)\leq \frac{r^d}{r_k}\leq r^{d-1}.$$
If $m<k$, by property $(4)$ we have $\dist(B^{m+1}_j, B^{m+1}_ i)\geq\sqrt[d]{\kappa_dr_mr_{m+1}^{d-1}}/2$. So, the disc $B(z,r)$ intersects at most $1+C\frac{r^d}{r_mr_{m+1}^{d-1}}$ discs $B^{m+1}_j$, for some constant $C$ depending on $d$. Applying property $(2)$ above, we get
$$\mu^k(B(z,r))=\sum_j\mu^k(B(z,r)\cap B^{m+1}_j)\leq \left(1+C\frac{r^d}{r_mr_{m+1}^{d-1}}\right)r_{m+1}^{d-1}\leq Cr^{d-1}.$$

From property (3) and the weak convergence we deduce that  for any disc $B(z,r)$, $\mu(B(z,r))\leq C_0r^{d-1}$ .

\section{Boundedness of $T_\mu(1)$ out of the support of $\mu$.}

We begin this section with some notation: since each $x\in E^k$ is contained in a unique disc $B^k_j$ and in a unique square $Q^k_j$, we shall denote them by $B^k(x)$ and $Q^k(x)$ respectively. 

The boundedness of $\|T_\mu(1)\|_{L^\infty(\C\setminus \supp(\mu))}$ will follow from the weak convergence of $\mu^k$ to $\mu$ and the following proposition.
\begin{prop}
Suppose that $\dist(x,\supp(\mu))=\varepsilon>0$ and $\sum_{k\geq 1}\delta_k^{\frac{1}{d}}<\infty$. Then for any $m\in\N$ with $\displaystyle{r_m<\frac{\varepsilon}{2d}}$, $$\left|\int_{\C}K(x-\xi)d\mu^m(\xi)\right|\leq C,$$
where $C$ is a constant that depends on the dimension $d$.

\end{prop}

\begin{proof}
Let $x^*\in\;\supp(\mu)$ be such that $\dist(x,x^*)=\varepsilon$ and fix $m$ with $r_m<\frac{\varepsilon}{2d}$. Let $q$ be the least integer with $r_q\leq\varepsilon$ (hence $m\geq q$). Write
\begin{equation}
 \begin{split}
  \int K(x-\xi)d\mu^{m}(\xi)&=\int_{B^{q}(x^*)}K(x-\xi)d\mu^{m}(\xi)+\sum_{k=1}^q\int_{B^{k-1}(x^*)\setminus B^k(x^*)}K(x-\xi)d\mu^{m}(\xi)\\&=A_1+A_2.
 \end{split}
\end{equation}
To estimate term $A_1$, notice that for any $\xi\in\;\supp(\mu)$, we have $\supp(\mu^m)\cap B^m(\xi)\neq\emptyset$. Therefore 
\begin{equation}\label{ep}
|x-\xi|\geq\dist(x,\supp(\mu^m))\geq\varepsilon-(1+\delta_{m+1})r_m\geq\varepsilon-\frac{\varepsilon}{d}\geq\frac{\varepsilon}{2}. 
\end{equation}
So, using property $(2)$ of the previous section, we get
$$|A_1|=\left|\int_{B^{q}(x^*)}K(x-\xi)d\mu^{m}(\xi)\right|\lesssim\left(\frac{2}{\varepsilon}\right)^{d-1}\mu^m(B^{q}(x^*))=\left(\frac{2}{\varepsilon}\right)^{d-1}r_q^{d-1}\leq 2^{d-1}.$$
To estimate term $A_2$, we claim that there exists $C=C(d)>0$ such that for any $k\in\N$ with $1\leq k\leq q$, 
\begin{equation}\label{claim}
\left|\int_{B^{k-1}(x^*)\setminus B^k(x^*)}K(x-\xi)d\mu^m(\xi)\right|\leq C\delta_k^{\frac{1}{d-1}}+C\sqrt[d]{\frac{\varepsilon}{r_{k-1}}}.
\end{equation}

If claim \eqref{claim} holds, then since $\delta_k=\sqrt{d}\sqrt[d]{\kappa_d(r_k/r_{k-1})^{d-1}}$ and by the definition of $q$, we have $r_q\leq\varepsilon$ and $r_k>\varepsilon$ for $1\leq k\leq q-1$, 
$$|A_2|\leq C\sum_{k=1}^{q}\delta_k^{\frac{1}{d-1}}+C\left(\sum_{k=1}^{q-1}\sqrt[d]{\left(\frac{r_k}{r_{k-1}}\right)^{d-1}}+\sqrt[d]{\frac{\varepsilon}{r_{q-1}}}\right)\leq C\left(\sum_{k=1}^{q}\delta_k^{\frac{1}{d-1}}+1\right)\leq C.$$
Let us mention again that all the constants, that appear in the proof, may depend on $d$ although sometimes we do not write the explicit dependence.

To show claim \eqref{claim}, we will use Lemma \ref{lemmabo}. Set $$\mathcal{A}=\{j:B_j^k\neq B^k(x^*)\;\mbox{ and }\;B_j^k\subset B^{k-1}(x^*)\}.$$ 
If $j\in\mathcal A$ and $\displaystyle{\dist(x,Q^k_j)\geq\frac 1 8\sqrt[d]{\kappa_dr_{k-1}r_k^{d-1}}}$, then we apply Lemma \ref{lemmabo} with $\displaystyle{\nu_1=\chi_{Q^k_j}\frac{m_d}{r_{k-1}}}$, $\displaystyle{\nu_2=\chi_{B^k_j}\mu^m}$, $R=r_{k-1}$, $r=r_k$ and 
$\displaystyle{x_0=x_{Q^k_j}}$. Hence,
\begin{equation}\label{first}
 \begin{split}
 \big|\int_{Q^k_j}K(x-\xi)\frac{dm_d(\xi)}{r_{k-1}}&-\int_{B^k_j}K(x-\xi)d\mu^m(\xi)\big|\\&\leq C\left(\frac{\sqrt[d]{r_{k-1}r_k^{d-1}}}{r_{k-1}}\int_{Q^k_j}\frac{dm_d(\xi)}{|x-\xi|^d}+r_k\int_{B^k_j}\frac{d\mu^m(\xi)}{|x-\xi|^d}\right).
 \end{split}
\end{equation}
Otherwise, namely if $j\in\mathcal A$ and $\displaystyle{\dist(x,Q^k_j)\leq\frac 1 8\sqrt[q]{\kappa_dr_{k-1}r_k^{d-1}}}$ (notice that there are at most $2^d$ pairwise disjoint cubes $Q^k_j$ in this situation, and it can only happen for $k=q$), then using \eqref{ep}, property $(2)$ of the previous section and the fact that for sets $A$ with finite measure, 
\begin{equation}\label{A}\int_A|K(\xi)|dm_d(\xi)\leq C\sqrt[d]{m_d(A)},
\end{equation} we get
\begin{equation}\label{second}
 \begin{split}
 \big|\int_{Q^k_j}K(z-\xi)&\frac{dm_d(\xi)}{r_{k-1}}-\int_{B^k_j}K(z-\xi)d\mu^m(\xi)\big|\\&\leq \frac{C}{r_{k-1}}\sqrt[d]{m_d(Q^k_j)}+\left(\frac{2}{\varepsilon}\right)^{d-1}\mu^m(B_j^k)\leq C\delta_k,
 \end{split}
\end{equation}
the last inequality coming from the fact that since $$\frac 1 8\sqrt[d]{\kappa_dr_{k-1}r_k^{d-1}}\geq\dist(x,Q^k_j)\geq\dist(x^*,Q^k_j)-\dist(x,x^*)\geq\frac 1 4\sqrt[d]{\kappa_dr_{k-1}r_k^{d-1}}-\varepsilon,$$ we have 
$\displaystyle{\varepsilon\geq\frac 1 8\sqrt[d]{\kappa_dr_{k-1}r_k^{d-1}}}$. 

Now write,
\begin{equation}
 \begin{split}
 \big|\int_{B^{k-1}(x^*)\setminus B^k(x^*)}&K(x-\xi)d\mu^m(\xi)\big|\leq A_3+A_4,
\end{split}
\end{equation}
where  $$A_3=\big|\int_{B^{k-1}(x^*)\setminus B^k(x^*)}K(x-\xi)d\mu^m(\xi)-\int_{\bigcup_{j\in\mathcal{A}}Q^k_j}K(x-\xi)\frac{dm_d(\xi)}{r_{k-1}}\big|$$ and $$A_4=\big|\int_{\bigcup_{j\in\mathcal{A}}Q^k_j}K(x-\xi)\frac{dm_d(\xi)}{r_{k-1}}\big|.$$ 
Using \eqref{first} and \eqref{second} we obtain
\begin{equation}
\begin{split}
A_3&\leq C\left(\delta_k+\sqrt[d]{\left(\frac{r_k}{r_{k-1}}\right)^{d-1}}\int_{B(x,2r_{k-1})\setminus B(x,\frac18\sqrt[d]{\kappa_dr_k^{d-1}r_{k-1}})}\frac{dm_d(\xi)}{|x-\xi|^d}+r_k\int_{\C\setminus B(x,\frac18\sqrt[d]{\kappa_dr_k^{d-1}r_{k-1}})}\frac{d\mu^m(\xi)}{|x-\xi|^d}\right)\\&\leq
C\left(\delta_k+\sqrt[d]{\left(\frac{r_k}{r_{k-1}}\right)^{d-1}}\log(\frac{r_{k-1}}{r_k})+\frac{r_k}{\sqrt[d]{r_k^{d-1}r_{k-1}}}\right)\\&\leq C\left(\delta_k+\delta_k\log(\frac 1{\delta_k})+\sqrt[d]{\frac{r_k}{r_{k-1}}}\right)\leq C\delta_k^{\frac 1{d}}.
\end{split}
\end{equation}
To estimate $A_4$, let $\displaystyle{\widetilde B^{k-1}(x^*)}$ denote the ball $\widetilde B^{k-1}_j$ containing $x^*$ and write
\begin{equation}
\begin{split}
A_4&\leq\big|\int_{\bigcup_{j\in\mathcal{A}}Q^k_j}K(x-\xi)\frac{dm_d(\xi)}{r_{k-1}}-\int_{\widetilde B^{k-1}(x^*)}K(x-\xi)\frac{dm_d(\xi)}{r_{k-1}}\big|\\&+\big|\int_{\widetilde B^{k-1}(x^*)}K(x-\xi)\frac{dm_d(\xi)}{r_{k-1}}\big|=A_{41}+A_{42}.
\end{split}
\end{equation}
Notice that $x\in \big(1+\frac{\varepsilon}{r_{k-1}}\big)B^{k-1}(x^*)$. Then, by the reflectionless property in Lemma \ref{reflectionless} and \eqref{A},
\begin{equation}
\begin{split}
A_{42}&=\big|\int_{\big(1+\frac{\varepsilon}{r_{k-1}}\big)B^{k-1}(x^*)\setminus\widetilde B^{k-1}(x^*)}K(x-\xi)\frac{dm_d(\xi)}{r_{k-1}}\big|\\&\leq\frac{C}{r_{k-1}}m_d\left(\big(1+\frac{\varepsilon}{r_{k-1}}\big)B^{k-1}(x^*)\setminus\widetilde B^{k-1}(x^*)\right)^{1/d}\leq C\sqrt[d]{\delta_k+\frac{\varepsilon}{r_{k-1}}}.
\end{split}
\end{equation}

To estimate term $A_{41}$, we use \eqref{A}. Hence,

$$A_{41}\leq \frac{C}{r_{k-1}}m_d\big(\bigcup_{j\in\mathcal{A}}Q^k_j\triangle\widetilde B^{k-1}(z^*)\big)^{1/d}\leq C\sqrt[d]{\delta_k}.$$

This finishes the proof of claim \eqref{claim} and the proposition.

\end{proof}

\section{Appendix}
For $n\in\N$, $n\geq 2$ and $z\in\C\setminus{\{0\}}$, we consider the family of kernels $$K_n(z)=\frac{P_{2n-1}(z)}{|z|^{2n}},$$ where
$P_{2n-1}$ is an homogeneous polynomial of degree $2n-1$. 

Decomposing the homogeneous polynomial $P_{2n-1}$ in spherical harmonics we get
\begin{equation}\label{decomposition}
\begin{split}
K_n(z)&=\frac{Q_{2n-1}(z)+Q_{2n-3}(z)|z|^2+\cdots+Q_1(z)|z|^{2n-2}}{|z|^{2n}}=\sum_{j=1}^{n}\frac{Q_{2j-1}(z)}{|z|^{2j}},
\end{split}
\end{equation}
where the $Q_{2j-1}$ are harmonic polynomials of degree $2j-1$. 
We consider now the family of kernels $$K_n^*(z)=\sum_{j=2}^{n}\frac{Q_{2j-1}(z)}{|z|^{2j}},$$ that is the ones for which the term $Q_1$ does not appear in the decomposition \eqref{decomposition}.
Since we can write each polynomial $Q_{2j-1}$ as $$Q_{2j-1}(z)=A_jz^{2j-1}+B_j\bar z^{2j-1},$$ for some constants $A_j$ and $B_j$, we have

$$K_n^*(z)=\sum_{j=2}^n\left( A_j\frac{z^{j-1}}{\bar z^j}+B_j\frac{\bar z^{j-1}}{z^j}\right)$$

Let $T_n$ and $T$ be the operators related to the kernels $K^*_n(z)$ and $K(z)=\frac{\bar z^{n-1}}{z^n}$, for $n\in\N$, $n\geq 2$ and $z\in\C\setminus\{0\}$, respectively. Clearly, Theorem \ref{maind} applies to both operators, and the proof follows  the same scheme as in \cite{jayenazarov}, as we did in the previous sections for the higher dimensional case. We want to include a different version of the proof of the reflectionless property in lemma \ref{reflectionless} for this case, without using the Fourier transform. This proof is based on lemma 3.1 in \cite{jayenazarov} but adapted to our family of  kernels. We find it interesting to include it here because in the plane it is not necessary to use the Fourier transform, the proof is just a short computation. 
\subsection{Proof of Lemma \ref{reflectionless} for the kernels $K(z)=\bar z^{n-1}/z^n$, $z\in\C\setminus\{0\}$.}

Recall that we want to show that for $z\in\C$, $r>0$ and any $\omega\in B(z,r)$, $$\int_{B(z,r)}K(\omega-\xi)dm_2(\xi)=0.$$

Without loss of generality, we may assume $z=0$ and $r=1$. 
Notice that 
$$K(\omega-\xi)=\frac{(\overline{\omega-\xi})^{n-1}}{(\omega-\xi)^n}=\frac{(-1)^n(\overline{\omega-\xi})^{n-1}}{\xi^n}\frac 1{\left(1-\frac{\omega}{\xi}\right)^n}.$$

Hence, for $|\omega|<|\xi|$, 

\begin{equation}
\begin{split}
K(\omega-\xi)&=\frac{(-1)^n(\overline{\omega-\xi})^{n-1}}{(n-1)!\;\xi^n}\sum_{k\geq n-1}k(k-1)\cdots(k-n+2)\left(\frac{\omega}{\xi}\right)^{k-n+1}\\
&=\frac{(-1)^n(\overline{\omega-\xi})^{n-1}}{(n-1)!\;\xi^n}\sum_{m\geq 0}(m+n-1)(m+n-2)\cdots(m+1)\left(\frac{\omega}{\xi}\right)^m\\
&=(-1)^n\sum_{k=0}^{n-1}\frac{\overline\omega^{n-1-k}\overline\xi^k}{k!(n-1-k)!}\sum_{m\geq 0}(m+n-1)(m+n-2)\cdots(m+1)\frac{\omega^m}{\xi^{m+n}},
\end{split}\end{equation}
which implies  $\displaystyle{\int_{|\xi|=t}K(\omega-\xi)dm_1(\xi)=0}$ for $|\omega|<t$, due to the fact that for  $k\neq l$ in $\Z$, we have 
\begin{equation}
\label{id1}
\int_{|\xi|=t}\bar\xi^k\xi^ldm_1(\xi)=0.
\end{equation}
For $|\omega|>|\xi|$, we write
\begin{equation}
\begin{split}
K(\omega-\xi)&=\frac{(-1)^n(\overline{\omega-\xi})^{n-1}}{(n-1)!\;\omega^n}\sum_{m\geq 0}(m+n-1)(m+n-2)\cdots(m+1)\left(\frac{\xi}{\omega}\right)^m\\
&=\sum_{k=0}^{n-1}\frac{(-1)^{n+k}\bar\xi^k\bar\omega^{n-1-k}}{k!(n-1-k)!\;\omega^n}\sum_{m\geq 0}(m+n-1)(m+n-2)\cdots(m+1)\left(\frac{\xi}{\omega}\right)^m.
\end{split}
\end{equation}
Integrating for $t<|\omega|$ and using \eqref{id1} we obtain
\begin{equation}
\begin{split}
&\int_{|\xi|=t}K(\omega-\xi)dm_1(\xi)\\&=\sum_{k=0}^{n-1}\frac{(-1)^{n+k}\bar\omega^{n-1-k}}{k!(n-1-k)!\;\omega^{n+k}}(k+n-1)(k+n-2)\cdots(k+1)\int_{|\xi|=t}|\xi|^{2k}dm_1(\xi)\\
&=2\pi\sum_{k=0}^{n-1}\frac{(-1)^{n+k}\bar\omega^{n-1-k}}{k!(n-1-k)!\;\omega^{n+k}}(k+n-1)(k+n-2)\cdots(k+1)t^{2k+1}.
\end{split}
\end{equation}
Therefore,
$$\int_0^{|\omega|}\int_{|\xi|=t}K(\omega-\xi)dm_1(\xi)dt=\frac{\pi\bar\omega^n}{\omega^{n-1}}\sum_{k=0}^{n-1}\frac{(-1)^{n+k}(k+n-1)(k+n-2)\cdots(k+2)}{k!(n-1-k)!}.$$
So, the desired conclusion follows if we show that 
\begin{equation}
\label{zero}
\sum_{k=0}^{n-1}\frac{(-1)^k(k+n-1)(k+n-2)\cdots(k+2)}{k!(n-1-k)!}=0.
\end{equation}
Consider the function $$G(x)=\sum_{k=0}^{n-1}\frac{(-1)^k(k+n-1)(k+n-2)\cdots(k+2)}{k!(n-1-k)!}x^{k+1}.$$
We will show that $G(1)=0$ and therefore \eqref{zero} holds.
Let $$f(x)=\frac{x^{n-1}(1-x)^{n-1}}{(n-1)!}=\sum_{k=0}^{n-1}\frac{(-1)^k}{k!(n-1-k)!}x^{k+n-1}.$$
Then, clearly $f^{n-2)}(x)=G(x)$ and hence $G(1)=0$.
\qed
\label{Bibliography}

\end{document}